\documentclass[letterpaper, 10 pt, conference]{ieeeconf}  

\IEEEoverridecommandlockouts
\usepackage{amsmath,amssymb,amsfonts}

\usepackage{bm}

\let\labelindent\relax
\usepackage{enumitem}
\usepackage[top=57pt, bottom=43pt, left=48pt, right=48pt]{geometry}

\usepackage{amsthm}

\theoremstyle{definition}
\newtheorem{definition}{Definition}[section]

\newtheorem{theorem}{Theorem}[section]

\newtheorem{remark}{Remark}[section]

\newtheorem{lemma}{Lemma}[section]

\usepackage{mathtools}

\usepackage{multicol}
\usepackage{multirow}
\usepackage{makecell}

\usepackage{pifont}
\newcommand{\cmark}{\text{\ding{51}}}%
\newcommand{\xmark}{\text{\ding{55}}}

\newcommand{\tp}{^{\mathrm{T}}}

\newcommand{\iv}{^{-1}}
\newcommand{\halfquad}{\mkern9mu}
\newcommand{\norm}[1]{\left\lVert#1\right\rVert}
\DeclarePairedDelimiter\ceil{\lceil}{\rceil}
\DeclarePairedDelimiter\floor{\lfloor}{\rfloor}

\makeatletter
\let\NAT@parse\undefined
\makeatother
\usepackage[noadjust]{cite}
\usepackage{xcolor}
\usepackage[hidelinks]{hyperref}
\usepackage{cleveref} 

\usepackage{cuted}
\usepackage{etoolbox}
\AfterEndEnvironment{strip}{\leavevmode}

\usepackage{dsfont}
\usepackage{nicefrac}
\usepackage{verbatim}

\usepackage{subcaption}
\usepackage{booktabs}
\newcommand{\edit}[1]{\textcolor{black}{#1}}
\usepackage{nth}
\usepackage{xfrac}

\begin{document}

\title{Polynomial and Parallelizable Preconditioning \\ for Block Tridiagonal Positive Definite Matrices \\
\thanks{This project is supported by the European Union’s 2020 Research and Innovation Programme (Marie Skłodowska-Curie Grant 953348 ELO-X) and the United States National Science Foundation (Awards 2246022, 2411369). Any opinions, findings, conclusions, or recommendations expressed in this material are those of the authors and do not necessarily reflect those of the funding organizations. Corresponding author contact: {\tt\footnotesize shaohui.yang@epfl.ch}}
}

\author{Shaohui Yang$^{1}$, Toshiyuki Ohtsuka$^{2}$, Brian Plancher$^{3}$, and Colin N. Jones$^{1}$
\thanks{$^{1}$Automatic Control Laboratory, EPFL, Lausanne, Switzerland.} 
\thanks{$^{2}$Department of Informatics, Kyoto University, Kyoto, Japan.}
\thanks{$^{3}$Barnard College, Columbia University, New York, USA.}
}

\maketitle

\begin{abstract}
The efficient solution of moderately large-scale linear systems arising from the KKT conditions in optimal control problems (OCPs) is a critical challenge in robotics. 
With the stagnation of Moore’s law, there is growing interest in leveraging GPU-accelerated iterative methods, and corresponding parallel preconditioners, to overcome these computational challenges. 
To improve the performance of such solvers, we introduce a parallel-friendly, parametrized multi-splitting polynomial preconditioner framework.
\edit{We first construct and prove the optimal parametrization theoretically in terms of the least amount of distinct eigenvalues and the narrowest spectrum range.
We then compare the theoretical time complexity of solving the linear system directly or iteratively. 
We finally show through numerical experiments how much the preconditioning improves the convergence of OCP linear systems solves.
}  
\end{abstract}

\section{Introduction}
The efficient solution of moderately large-scale linear systems arising from the Karush-Kuhn-Tucker (KKT) conditions in optimal control problems (OCPs) is a fundamental challenge in model predictive control (MPC) and trajectory optimization~\cite{betts1998survey}. These problems are central to enabling real-time, high-performance robotic behaviors in tasks ranging from locomotion to manipulation~\cite{hogan2018reactive,grandia2023perceptive}. \edit{While these systems are typically characterized by a block tridiagonal positive definite matrix~\cite{adabag2024mpcgpu,bu2024symmetric}, direct factorization is computationally prohibitive for large problem instances.} This necessitates the development of custom solvers optimized for scalable efficiency on their target computational platforms~\cite{adabag2024mpcgpu,nguyen2024tinympc}.

At the same \edit{time}, as Moore’s law slows, traditional CPU performance scaling has stagnated~\cite{esmaeilzadeh2011dark}, \edit{increasing interest in algorithms amenable to acceleration on parallel computational hardware (e.g, GPUs). This has led to increased use of iterative methods like the Preconditioned Conjugate Gradient (PCG)~\cite{eisenstat1981efficient,helfenstein2012parallel,schubiger2020gpu,adabag2024mpcgpu}, whose performance is dependent on the quality of preconditioners~\cite{saad2003iterative}.}
Unfortunately, many popular preconditioners place limitations on the underlying matrix structure \edit{(e.g., non-negativity of off-diagonal entries~\cite{yun1998block}, diagonal dominance~\cite{concus1985block}, Toeplitz structure~\cite{di1993cg}),
or are not inherently parallel-friendly (e.g., block incomplete factorization~\cite{yun1998block,concus1985block}, SDP-based preconditioners~\cite{qu2024optimal}), limiting their applicability for parallel computation of OCP KKT systems.}

To address these challenges, we propose a novel, \edit{parallel-friendly}, polynomial preconditioner tailored for symmetric positive definite block tridiagonal matrices. 
\edit{Our key contributions are:} 
(1) Development of a parametrized family of polynomial preconditioners and rigorous analysis of their \edit{spectrum and positive definiteness; 
(2) Optimal parameters that reduce the eigenvalue multiplicity by 50\% and produce the most compact spectrum possible for the preconditioned system; 
(3) Comparison of time complexity between Cholesky factorization and PCG;  
and (4) Numerical validation indicating reduced PCG iteration and matrix-vector multiplication counts versus the state-of-the-art preconditioners.}

\section{Preliminaries}
\subsection{Symmetric positive definite block tridiagonal matrix}
We focus on block tridiagonal matrices\footnote{\textit{Notation:} Symmetric positive definite matrices will be abbreviated as s.p.d.. 
A symmetric matrix with positive eigenvalues is described as p.d..
A non-symmetric matrix is p.d. if its symmetric part--$sym(A) \coloneq \frac{1}{2}(A + A\tp)$--is p.d..  
$S^{n}_{++}$ denotes the set of s.p.d. matrices of size $n \times n$. 
$\rho(A)$ denotes the spectral radius and $\sigma(A)$ denotes the spectrum of $A$. 
If $A$ is symmetric, then $\mathcal{I}_{A}$ denotes the minimal interval that contains all its real eigenvalues.} $A \in S^{Nn }_{++}$, where $N$ denotes the number of diagonal blocks and $n$ denotes the block size. \edit{A minimal example of $N=3$ is used across the majority of the paper: } 
\begin{equation}\label{eq:def-spd-block-tri}
    A = \begin{bmatrix}
        D_1 & O_1 & \\
        O_1\tp & D_2 & O_2 \\
        & O_2\tp & D_3
    \end{bmatrix}, \quad D_k, O_k \in \mathbb{R}^{n \times n}.
\end{equation}
An immediate result of $A \in S^{Nn }_{++}$ is $D_k \in S_{++}^{n}, \forall k$. No assumption is placed towards sparsity of $D_k, O_k$, i.e., they are viewed as general dense blocks with $n^2$ entries each. 



\subsection{Matrix splittings and polynomial preconditioning}

Many preconditioners arise from matrix splittings. In this section, we \edit{first} review the definitions of \edit{{\it convergent} splitting}, {\it P-regular} splitting, and {\it multi-}splitting: 
\begin{definition}\cite{ortega1990numerical}\label{def:convergent-splitting}
    Let $A, B, C \in \mathbb{R}^{n \times n}$. $A = B- C$ is a \textit{convergent} splitting if $B$ is nonsingular and $\rho(B\iv C) < 1$. 
\end{definition}
\begin{definition}\cite{ortega1990numerical}\label{def:p-regular-splitting}
    Let $A, B, C \in \mathbb{R}^{n \times n}$. $A = B- C$ is a \textit{P-regular} splitting of $A$ if $B$ is nonsingular and $B+C$ is p.d..  
\end{definition}

\begin{definition}\cite{o1985multi}\label{def:multi-splitting}
    Let $A, B_k, C_k, \edit{W_k} \in \mathbb{R}^{n \times n}$. If   
    \begin{itemize}
        \item $A = B_k - C_k, k=1, \dots, K$ with each $B_k$ invertible and
        \item $\sum_{k=1}^K \edit{W_k} = I$ with \edit{diagonal weight matrix $W_k$},  
    \end{itemize}
    then $(B_k, C_k, \edit{W_k})$ is called a \textit{multi-}splitting of $A$. 
\end{definition}

Given a multi-splitting, \edit{a matrix pair $(G,H)$ is constructed where $H \coloneqq \sum_{k} W_k B_k\iv C_k, G \coloneqq \sum_{k} W_k B_k\iv$ and manipulated as a new} single splitting $A = B - C$ \edit{where $B = G\iv, C = G\iv H$. It is preferable to express the reverse:}
\begin{align}\label{eq:def-GH-from-BC}
    G &= B\iv, & H &= B\iv C. 
\end{align}

\edit{P-regular splittings are convergent~\cite{ortega1990numerical}. If a multi-splitting is weighted nonnegatively by a set of P-regular splittings, then the equivalent single splitting is convergent~\cite{o1985multi}. Convergence of splitting is necessary in the definition of Neumann series and $m$-step polynomial preconditioner presented below.}

\begin{lemma}[Neumann series]\cite{neumann1877untersuchungen}
\label{lem:Neumann-series}
    \edit{If $A = B- C$ is a convergent splitting of the nonsingular matrix $A$}, then with~\eqref{eq:def-GH-from-BC}, 
    \begin{equation}
        \lim_{k \rightarrow\infty} \Big( I + H + H^2 + \dots + H^k \Big) G = A\iv.         
    \end{equation}
\end{lemma}

\begin{definition}[Truncated Neumann series]\cite{adams1985m}
    Under the same conditions as~\Cref{lem:Neumann-series}, to approximate \edit{$A\iv$}, the \edit{inverse} of $m$-step polynomial preconditioner is defined as:
    \begin{equation}\label{eq:def-m-step-polynomial-preconditioner-inverse}
        M_m\iv \coloneq (I + H + H^2 + \dots + H^{m-1}) G \approx A\iv. 
    \end{equation}
\end{definition}

\edit{Though $M_m$ is the polynomial preconditioner, $M_m\iv$ is preferred by PCG because it avoids the expensive back solve.} While convergence of the splitting is \edit{necessary}, it is insufficient for $M_m$ to be a s.p.d. preconditioner for the CG method. As such, we present \edit{the following lemma} to fill that gap:
\begin{lemma}\cite{adams1985m}
\label{lem:polynomial-preconditioner-spd}
    If $A = B- C$ is s.p.d. and $B$ is symmetric and nonsingular, then with~\eqref{eq:def-GH-from-BC} and~\eqref{eq:def-m-step-polynomial-preconditioner-inverse}, 
    \begin{enumerate}
        \item $M_m$ is symmetric. 
        \item \edit{$\rho(B\iv C) < 1$ if and only if $B + C$ is p.d.. }
        \item For $m$ odd, $M_m$ is p.d. if and only if $B$ is p.d.. 
        \item For $m$ even, $M_m$ is p.d. if and only if $B + C$ is p.d.. 
    \end{enumerate}
\end{lemma}

\subsection{Splittings of s.p.d. block tridiagonal matrix}
We examine two fundamental splittings of $A$ in~\eqref{eq:def-spd-block-tri}, which serve as the basis for \edit{the upcoming family of} multi-splittings.

\subsubsection{Diagonal splitting}
The diagonal splitting $A = B_d - C_d$ is defined as: 
\begin{equation}\label{eq:def-diagonal-splitting}
B_d \coloneq \begin{bmatrix}
    D_1 & \\
    & D_2 & \\
    & & D_3
\end{bmatrix}, \hspace{0.1em}
C_d \coloneq \begin{bsmallmatrix}
    0 & -O_1 & \\
    -O_1\tp & 0 & -O_2 \\
    & -O_2\tp & 0
\end{bsmallmatrix}. 
\end{equation}
\begin{lemma}
    The diagonal splitting is P-regular.
\end{lemma}
\begin{proof}
    $B_d$ is s.p.d. and thus invertible. $B_d + C_d$ is symmetric so it suffices to prove it is p.d.. 
    \edit{Take $N=3$ as an example. }$A$ is s.p.d. $\Rightarrow \forall x = \begin{bsmallmatrix} x_1\tp & x_2\tp & x_3\tp \end{bsmallmatrix}\tp \neq 0, x\tp A x > 0$. 
    Define vectors $x_{e} = \begin{bsmallmatrix} x_1\tp & -x_2\tp & x_3\tp \end{bsmallmatrix}\tp, x_{o} = \begin{bsmallmatrix} -x_1\tp & x_2\tp & -x_3\tp \end{bsmallmatrix}\tp$. 
    By construction, $x\tp A x = x_{e}\tp (B_d + C_d) x_{e} = x_{o}\tp (B_d + C_d) x_{o} > 0$. 
\end{proof}

\subsubsection{Stair splittings}
The left and right stair splittings~\cite{lu1999stair} $A = B_l - C_l = B_r - C_r$ are defined as: 
\begin{equation}\label{eq:def-stair-splittings}
\begin{aligned}
    B_l &\coloneq \begin{bmatrix}
        D_1 &  & \\
        O_1\tp & D_2 & O_2 \\
        &  & D_3 
    \end{bmatrix}, \halfquad 
    C_l \coloneq \begin{bmatrix}
        0 & -O_1 & \\
         & 0 &  \\
        & -O_2\tp & 0
    \end{bmatrix}. \\
    B_r &\coloneq \begin{bmatrix}
        D_1 & O_1 & \\
         & D_2 &  \\
        & O_2\tp & D_3
    \end{bmatrix}, \halfquad 
    C_r \coloneq \begin{bmatrix}
        0 &  & \\
        -O_1\tp & 0 & -O_2 \\
        & & 0
    \end{bmatrix}.
\end{aligned}
\end{equation}
\begin{lemma}
    The left and right stair splittings are P-regular. 
\end{lemma}
\begin{proof}
    $B_l, B_r$ are invertible because they contain all the diagonal s.p.d. blocks. $sym(B_l + C_l) = sym(B_r + C_r) = B_d $, which is s.p.d., holds by construction.
\end{proof}

\section{m-step preconditioners from multi-splitting}
In this section, we construct a family of multi-splittings \edit{by weighting} diagonal splitting~\eqref{eq:def-diagonal-splitting} and stair splittings~\eqref{eq:def-stair-splittings} \edit{of $A$ in~\eqref{eq:def-spd-block-tri} parametrically}. 
\edit{Each multi-splitting corresponds to a matrix pair $(G, H)$ that is extendable to $M_m\iv$ in~\eqref{eq:def-m-step-polynomial-preconditioner-inverse}. 
We first analyze how the weightings influence the eigenvalue distribution of $M_m\iv A$. 
We then prove conditions under which $M_m\iv$ qualifies as a preconditioner for CG. 
We conclude with an optimal set of parameters resulting in the most clustered spectrum for faster PCG convergence}. 

\subsection{Parametric multi-splitting family}
Following~\Cref{def:multi-splitting}, a family of multi-splittings is constructed with $(B_k, C_k, \edit{W_k}), k \in \{d,l,r\}, K=3$ and $\edit{W_l = W_r} = aI, \edit{W_d} = bI$ \edit{where the weights $(a, b)$ belong to the set $\mathcal{C} \coloneq \{(a, b) \in \mathbb{R} \mid 2a + b = 1\}$. 
Each multi-splitting corresponds to a parametric matrix pair $(G_{ab}, H_{ab})$: }
\begin{equation}\label{eq:def-GH-ab-multi-splitting}
\begin{aligned}
    G_{ab} &\coloneqq a(B_l\iv + B_r\iv) + bB_d\iv, \\
    H_{ab} &\coloneqq a(B_l\iv C_l + B_r\iv C_r) + bB_d\iv C_d. 
\end{aligned}
\end{equation}

\eqref{eq:def-GH-ab-multi-splitting} can be interpreted as deriving from a single splitting: 
\begin{align}\label{eq:def-ab-multi-splitting}
    A &= B_{ab} - C_{ab}, &
    B_{ab} &\coloneqq G_{ab}\iv, & 
    C_{ab} &\coloneq G_{ab}\iv H_{ab}. 
\end{align}

The inverse of $m$-step preconditioner related to~\eqref{eq:def-ab-multi-splitting} is: 
\begin{equation}\label{eq:Mabinv}
    _{ab}M_m\iv \coloneqq (I + H_{ab} + H_{ab}^2 + \dots + H_{ab}^{m-1}) G_{ab}. 
\end{equation}

\edit{$_{ab}M_m\iv$ is the focus from now on. We will discuss its spectrum and symmetric positive definiteness based on parameters $a, b, m$. An established fact is that if $a, b \geq 0$, then $G_{ab}$ is s.p.d. and $\rho(H_{ab})<1$~\cite{o1985multi}. We will explore beyond that.}

\subsection{Spectrum analysis}
\edit{
In this subsection, we will conclude that the eigenvalues of $_{ab}M_m\iv A$ are functions of that of $B_l\iv C_l$. 
To start with the base case $m=1$, 
the following notations are introduced}: for $v\tp = (v_1\tp, \dots, v_N\tp) \in \mathbb{R}^{Nn}, v_i \in \mathbb{R}^n$, denote $v_e\tp = (0, v_2\tp, \dots, 0, v_{2j}\tp, \dots)$ and $v_o\tp = (v_1\tp, 0, \dots, v_{2j+1}\tp, 0, \dots)$ such that $v = v_e + v_o$. 
$f_{a}(\lambda)$ denotes $f_{a+}(\lambda)$ or $f_{a-}(\lambda)$, where 
\begin{equation}
    \begin{aligned}
        f_{a\pm}(\lambda) &\coloneqq a\lambda \pm (1-a)\sqrt{\lambda}. 
    \end{aligned}
\end{equation}

\begin{lemma}\label{lem:Hab-GabA-eigenpairs}
    If $(\lambda \neq 0, v = v_e + v_o)$ is an eigenpair of $B_l^{-1} C_l$, 
    \begin{enumerate}
        \item \edit{then $(f_{a\pm}(\lambda), v_e \pm \sqrt{\lambda}v_o)$} are eigenpairs of $H_{ab}$. 
        \item \edit{then $(1 - f_{a\pm}(\lambda), v_e \pm \sqrt{\lambda}v_o)$} are eigenpairs of $G_{ab} A$.
    \end{enumerate}
\end{lemma}

\begin{proof}
    By construction, $G_{ab}A + H_{ab} = I, \forall (a, b) \in \mathcal{C}$. 
    So Point 2) follows from Point 1) and vice versa. 
    By inspection, $C_l v_o = C_r v_e = 0$. According to~\cite{bu2024symmetric}, 
    \begin{align}
        B_d\iv C_l v &= B_d\iv C_l v_e = \lambda v_o, \qquad 
        B_d\iv C_r v = B_d\iv C_r v_o = v_e, \notag \\ 
        B_l\iv C_l v &= B_l\iv C_l v_e = \lambda (v_e + v_o), \\ 
        B_r\iv C_r v &= B_r\iv C_r v_o = v_e + \lambda v_o. \notag 
    \end{align}
    A reasonable ansatz of the eigenvector of $H_{ab}$ would be $v_{ab} = \alpha v_e + \beta v_o$. We expand $H_{ab} v_{ab}$ to further validate our guess: 
    \begin{align}\label{eq:Hab-eigenvalue-proof}
        &\hspace{1.2em} H_{ab} v_{ab} = [ a (B_l\iv C_l + B_r\iv C_r) + (1-2a)B_d\iv C_d ] v_{ab} \notag \\
        &= a[(\alpha \lambda + \beta) v_e + (\alpha + \beta) \lambda v_o] + (1-2a) (\beta v_e + \alpha \lambda v_o ) \notag \\
        &= (a \alpha \lambda - a \beta + \beta) v_e + (a \beta \lambda + \alpha \lambda - a \alpha \lambda) v_o \\
        &= \lambda_{ab} v_{ab} = \lambda_{ab} \alpha v_e + \lambda_{ab} \beta v_o. \notag
    \end{align}
    the last parts of~\eqref{eq:Hab-eigenvalue-proof} hold if and only if $\lambda_{ab} = \frac{a \alpha \lambda - a \beta + \beta}{\alpha} = \frac{a \beta \lambda + \alpha \lambda - a \alpha \lambda}{\beta} \Leftrightarrow (1-a) (\alpha^2 \lambda -\beta^2)= 0$. If $a \neq 1$, then $\beta = \pm \sqrt{\lambda} \alpha$ must hold. 
    By setting $\alpha = 1$, Point 1) is proven. If $a = 1$, then $(\alpha, \beta)$ are unrestricted. Hence $v_e, v_o$ are eigenvectors and $f_{a\pm}(\lambda)$ both degenerate to the identity function. 
\end{proof}
%
\edit{\Cref{lem:Hab-GabA-eigenpairs} concludes that $G_{ab} A$ and $H_{ab}$ share the same eigenvectors, a fact that generalizes to $_{ab}M_m\iv A$.}
\begin{lemma}\label{lem:eigenvalues-Mabinv*A}
    If $\lambda \neq 0$ is an eigenvalue of $B_l\iv C_l$, \edit{then} $_{ab}M_m\iv A$ has a pair of eigenvalues at \edit{$1 - f_{a\pm}(\lambda)^m$}. 
\end{lemma}
\begin{proof}
     Let $v$ be an eigenvector of $G_{ab}A$ and $H_{ab}$. $\forall j \in \mathbb{N}$: 
    \begin{equation}
        H_{ab}^j G_{ab} A v = (1-f_{a}(\lambda)) f_a(\lambda)^j v.
    \end{equation}
    Recall the definition of $_{ab}M_m\iv A$ by~\eqref{eq:Mabinv}:
    \begin{align}\label{eq:m-step-eigenvalues}
        _{ab}M_{m}\iv A v &= (I + H_{ab} + H_{ab}^2 + \dots + H_{ab}^{m-1}) G_{ab} A v \notag \\
        &= (1-f_{a}(\lambda)) (1 + f_a(\lambda) + \dots + f_a(\lambda)^{m-1}) v \notag \\
        &= (1- f_a(\lambda)^m) v.
    \end{align}
    \eqref{eq:m-step-eigenvalues} follows from the sum of geometric progression. 
\end{proof}

\edit{\Cref{lem:Hab-GabA-eigenpairs,lem:eigenvalues-Mabinv*A} reveal that the eigenvalues of $H_{ab}$, $G_{ab} A$, and $_{ab}M_m\iv A$ are ``generated'' from the eigenvalues of $B_l\iv C_l$.
A natural question is on the spectrum range of $B_l\iv C_l$, which is answered by the following lemma. }

\begin{lemma}\label{lem:spectrum-of-BlinvCl}
    \edit{$\forall \lambda \in \sigma(B_l\iv C_l), \lambda \neq 0 \Rightarrow \lambda \in \mathbb{R}, \lambda \in (0,1)$. }
\end{lemma}
\vspace{-1em}
\begin{proof}
    \edit{$B_d$} is s.p.d, so \edit{$X = B_d^{-\frac{1}{2}}$} is well-defined and s.p.d.. \edit{$Y = B_d\iv A$} is similar to the matrix \edit{$Z = X A X$} because $X\iv \edit{Y} X = \edit{Z}$, so \edit{$Y, Z$} has the same sets of eigenvalues. \edit{$Z$} is congruent to $A$ so only has positive eigenvalues. Hence all eigenvalues of \edit{$Y$} are real and positive. 
    \edit{According to~\Cref{lem:Hab-GabA-eigenpairs}, if $\lambda \neq 0$ is an eigenvalue of $B_l\iv C_l$, then $1 - f_{a\pm}(\lambda) = 1 \mp \sqrt{\lambda}$ are the eigenvalues of $Y = G_{0,1}A$ with $a=0$. Finally, $1 \mp \sqrt{\lambda} \in \mathbb{R}_{>0} \Rightarrow \lambda \in \mathbb{R}, \lambda \in (0,1)$. } 
\end{proof}

\begin{remark}\label{rmk:big-m-good-preconditioner}
    \Cref{lem:eigenvalues-Mabinv*A} \edit{illustrates that when $a, b$ are fixed}, larger $m$ generates $_{ab}M_{m}\iv$ closer to $A\iv$. 
    \edit{According to~\Cref{lem:spectrum-of-BlinvCl}, $\forall \lambda \in \sigma(B_l\iv C_l), \lambda \in [0,1)$. 
    So,} $\forall a \in [0,1], f_a(\lambda) \in (0,1) \Rightarrow \lim_{m \rightarrow\infty} \edit{1-f_a(\lambda)^m = 1}$. Big $m$ pushes eigenvalues to one, so $_{ab}M_{m}\iv A$ tends to identity. 
    For fixed $m$, the smaller $f_a(\lambda)$, the \edit{faster $1-f_a(\lambda)^m$ goes to $1$}. 
\end{remark}

\subsection{Symmetric positive definiteness analysis}
\edit{In this subsection, we first discuss how $a$ influences the convergence of matrix splitting $A = B_{ab} - C_{ab}$ and the spectrum of $G_{ab}A$. From this we formally prove the necessary and sufficient conditions for $_{ab}M_m\iv$ being s.p.d.. We then present two cases appeared in the literature. We conclude with an optimal pair of $(a, b)$ in terms of clustered spectrum.}
\begin{lemma}\label{lem:a-b-negative-rhoH-spectrum-of-GabA}
\edit{Assume that $\lambda \in \sigma(B_l\iv C_l)$ may take arbitrary value within $[0,1)$. 
Let $\lambda_{max}$ denote $\max_{\lambda} \sigma(B_l\iv C_l)$. }
\begin{enumerate}[leftmargin=*]
    \item \edit{$\rho(H_{ab}) < 1$ if and only if $a \in [0,1]$}.
    \begin{enumerate}
        \item $\rho(H_{ab}) = \max(1-2a, f_{a+}(\lambda_{max}))$ if $a \in [0, \frac{1}{3}]$. 
        \item $\rho(H_{ab}) = \max(\frac{(1-a)^2}{4a}, f_{a+}(\lambda_{max}))$ if $a \in (\frac{1}{3}, 1]$. 
        \item If $a \in (-\infty, 0) \cup (1, \infty)$, there \edit{exists} some $\lambda \in \sigma(B_l\iv C_l)$ that lead to $\rho(H_{ab}) > 1$. 
    \end{enumerate}
    \item \edit{All eigenvalues of $G_{ab}A$ are positive, i.e., $\mathcal{I}_{G_{ab}A} \subset \mathbb{R}_{>0}$, if and only if $a \in [-1,1]. $}
    \begin{enumerate}
        \item If $a \in [-1, \frac{1}{3}]$, then $\mathcal{I}_{G_{ab}A} = (0, 2-2a)$.
        \item If $a \in (\frac{1}{3}, 1]$, then $\mathcal{I}_{G_{ab}A} = (0, 1+\frac{(1-a)^2}{4a})$. 
        \item If $a \in (-\infty, -1) \cup (1, \infty)$, there \edit{exists} some $\lambda \in \sigma(B_l\iv C_l)$ that leads to negative eigenvalue of $G_{ab}A$. 
        \end{enumerate}
\end{enumerate}
\end{lemma}
\begin{proof}
See Appendix~\ref{appendix:A-proof}. 
\end{proof}

\edit{As a result of~\Cref{lem:a-b-negative-rhoH-spectrum-of-GabA}, convergence of matrix splitting $A = B_{ab} - C_{ab}$ in~\eqref{eq:def-ab-multi-splitting} and positive definiteness of $G_{ab} A$ hold iff $a \in [0,1]. $}
A new set \edit{$\mathcal{C}_g \coloneq \{a,b \in \mathbb{R} \mid 2a+b = 1, a \geq 0, b\geq -1\}$ is introduced considering such restriction. } 

\edit{The following theorem proves that $\forall (a, b) \in \mathcal{C}_g, \forall m \in \mathbb{N}$, $_{ab}M_m$ or $_{ab}M_m\iv$ in~\eqref{eq:Mabinv} is s.p.d. and qualifies as a preconditioner for the conjugate gradient method. }

\begin{theorem}\label{thm:negative-ab-multi-splitting}
Consider $\forall (a, b) \in \mathcal{C}_g$, 
\begin{enumerate}
    \item $G_{ab}$ is s.p.d..
    \item The splitting~\eqref{eq:def-ab-multi-splitting} is convergent, i.e., $\rho(H_{ab}) < 1$. 
    \item $\forall m \in \mathbb{N}$, the matrix $_{ab}M_m\iv$ in~\eqref{eq:Mabinv} is s.p.d..
\end{enumerate}
\end{theorem}
\begin{proof}
    According to~\Cref{lem:a-b-negative-rhoH-spectrum-of-GabA}, $(a,b) \in \mathcal{C}_g \Rightarrow$ eigenvalues of $X = G_{ab}A$ are positive. $G_{ab} = X A\iv$. Multiplying $G_{ab}$ by $A^{\frac{1}{2}}$ on both sides  
    $\Rightarrow A^{\frac{1}{2}} G_{ab} A^{\frac{1}{2}} = A^{\frac{1}{2}} X A^{-\frac{1}{2}}$.
    The right matrix is similar to $X$ and hence has the same eigenvalues, while the left matrix is congruent to $G_{ab}$ and hence has the same number of positive eigenvalues. Thus $G_{ab}$ has the same number of positive eigenvalues as $X$ and so is p.d.. \edit{By construction, $G_{ab}$ is symmetric, so Point 1) is proven.
    Point 2) follows from~\Cref{lem:a-b-negative-rhoH-spectrum-of-GabA} directly. }
    \edit{According to~\Cref{lem:a-b-negative-rhoH-spectrum-of-GabA,lem:polynomial-preconditioner-spd}, $(a,b) \in \mathcal{C}_g \Rightarrow \rho(H_{ab}) < 1 \Rightarrow B_{ab} + C_{ab}$ is p.d.. For $m$ even, Point 3) follows from Points 1,4) of~\Cref{lem:polynomial-preconditioner-spd}.
    For $m$ odd, Point 3) follows from Points 1,3) of~\Cref{lem:polynomial-preconditioner-spd}.}
\end{proof}

\subsection{Example cases of parametric multi-splitting}
One optimal and two extreme cases of $(G_{ab}, H_{ab})$ are presented. Their sparsity patterns are visualized by~\Cref{fig:sparsity}. 

\subsubsection{Optimal case \texorpdfstring{$a=1, b=-1$}{a=1, b=-1}}\label{subsubsec:optimal}

$H_{1,-1} \edit{\coloneq H_{opt}}$ is block pentadiagonal with zero super and sub-diagonal blocks.
The block tridiagonal matrix $G_{1,-1} \edit{\coloneq G_{opt}}$ \edit{is called} ``symmetric stair preconditioner''~\cite{bu2024symmetric}, where \edit{$H_{opt}$ and $_{opt}M_m\iv$ were} not defined and positive definiteness of $\edit{G_{opt}}$ was not proven. 
\edit{Optimality in terms of spectrum clustering (eigenvalue multiplicity and spectrum range) is presented below. }

\begin{theorem}\label{thm:optimal-spectrum}
    If $a = 1, b = -1$, then $_{ab}M_m\iv$, $G_{ab}$, and $H_{ab}$ have the following optimal properties: 
    \begin{enumerate}
        \item $_{ab}M_m\iv A$ bears the least amount of distinct eigenvalues. 
        \item The interval $\mathcal{I}_{_{ab}M_m\iv A}$ has the smallest length.
    \end{enumerate}
\end{theorem}
\vspace{-0.8em}
\begin{proof}
    Recall it was proven in~\Cref{lem:eigenvalues-Mabinv*A} that if $0 \neq \lambda \in \sigma(B_l\iv C_l)$, then it ``generates'' a pair of eigenvalues of $_{ab}M_m\iv A$ at $1 - f_{a\pm}(\lambda)^m$. 
    A special case is when $a = 1 \Rightarrow f_{a+} = f_{a-} = \mathds{1}$ leading to two identical eigenvalues. 
    The sparsity pattern of $B_l\iv C_l$ determines that it has $\ceil*{\frac{N}{2}}n$ eigenvalues at $0$ and $\floor*{\frac{N}{2}}n$ non-zero eigenvalues~\cite{bu2024symmetric}. 
    For $N = 2k$, $kn$ non-zero $\lambda \in \sigma(B_l\iv C_l)$ ``generate'' $2kn$ eigenvalues of $_{ab}M_m\iv A$, same as the matrix dimension. 
    For $N=2k-1$, $(k-1)n$ non-zero $\lambda$ ``generate'' $(2k-2)n$ eigenvalues. The rest $n$ eigenvalues at $1$ are mapped from $0 \in \sigma(B_l\iv C_l)$. 
    \Cref{tab:N-even-GH-a=1-b=-1,tab:N-odd-GH-a=1-b=-1} summarize the spectrum of matrices arising from $(G_{1,-1}, H_{1,-1})$ for even and odd $N$. 
    $\forall m \in \mathbb{N}, {}_{ab}M_m\iv A$ has $d_{Nna}$ distinct eigenvalues\footnote{\label{ft:inspection}\edit{By inspection, $\forall m \in \mathbb{N}, {}_{1,-1}M_{m}\iv = {}_{0,1}M_{2m}\iv$. So the first case of~\eqref{eq:distinct-eigenvalues-count} also holds for $a=0$ and $m$ is even. This fact is footnoted for conciseness. }}, where
    \begin{equation}\label{eq:distinct-eigenvalues-count}
    d_{Nna} \coloneq
        \begin{cases}
           \frac{N}{2}n \hspace{0.5em} \text{ or } \hspace{0.5em}
           \floor*{\frac{N}{2}}n + 1,
           & \text{ if $a=1$.} \\
           Nn \hspace{0.5em} \text{ or } \hspace{0.5em}
           2\floor*{\frac{N}{2}}n + 1, 
           & \text{ if $a \in [0,1)$.}
        \end{cases}
    \end{equation}
    According to~\Cref{lem:eigenvalues-Mabinv*A}, the eigenvalues of $_{ab}M_{m}\iv A$ are $1 - f_a(\lambda)^m$. 
    As detailed in~\Cref{appendix:A-proof}, $f_{a+} \in (0,1)$ but the range of $f_{a-}$ depends on $a$. 
    $\max_{\lambda} f_a^m = f_{a+}^m(\lambda_{max})$, which achieves its minimum $\lambda_{max}^m$ at $a=1$ because $\sqrt{\lambda} > \lambda, \forall \lambda \in (0,1)$, so the smallest eigenvalue of $_{1,-1}M_m\iv A$ is the biggest among $\mathcal{C}_g$. 
    If $m$ is even, then $\min_{\lambda} f_a^m = 0, \forall a \in [0,1]$. If $m$ is odd, then $\min_{\lambda} f_a^m = \min_{\lambda} f_{a-}^m$, which is negative if $a \in [0, 1)$ and zero if $a=0$. 
    So the biggest eigenvalue of $_{1,-1}M_m\iv A$ is the smallest among $\mathcal{C}_g$. 
\end{proof}
\subsubsection{Extreme case \texorpdfstring{$a = 0, b = 1$}{a=0, b=1}}\label{subsubsec:block-jacobi}
The block diagonal matrix $G_{0, 1}$ is the block Jacobi preconditioner~\cite{saad2003iterative}. 
\edit{$H_{0,1}$ and the subsequent ${}_{0,1}M_2\iv$ were used in~\cite{dubois1979approximating}. 
Compared with $G_{0,1} (m=1)$, nearly 50\% reduction in PCG iteration counts was reported but not explained. 
Such phenomenon can now be addressed by~\cref{ft:inspection}: ${}_{opt}M_1\iv = {}_{0,1}M_2\iv$, so the number of distinct eigenvalues of ${}_{0,1}M_2\iv A$ is halved. 
}

\subsubsection{Extreme case \texorpdfstring{$a = \frac{1}{2}, b = 0$}{a=0.5, b=0}}\label{subsubsec:symmetric-stair}
The block tridiagonal matrix $G_{\nicefrac{1}{2},0}$ \edit{is called} ``additive stair preconditioner'' in~\cite{bu2024symmetric}. 

\begin{table}[!tb]
\centering
\caption{Spectrum of matrices related to $(G_{\textcolor{black}{opt}}, H_{\textcolor{black}{opt}}) \in \mathbb{R}^{Nn \times Nn}$}
\vspace{-0.5em}
\begin{subtable}{\linewidth}
\caption{$N = 2k$ \edit{(even)} }
\vspace{-0.5em}
\scalebox{0.92}{
\begin{tabular}{ |c|c|c|c|c|c| } 
 \hline
 \multirow{2}{*}{Matrix} & \multirow{2}{*}{\makecell{\# Exact \\ Zeros}}& \multicolumn{3}{c|}{Eigenvalues in $(0, 1)$} & \multirow{2}{*}{\makecell{\# Exact \\ Ones}}  \\ 
 \cline{3-5}
 & & \# Distinct & Pair? & Example &  \\
 \hline
 $B_l\iv C_l$ & $k n$ & $k n$ & \xmark & $\lambda$ & 0 \\ 
 \hline
 $H_{\textcolor{black}{opt}}$ & 0 & $k n$ & \cmark & $\lambda$ & 0  \\ 
 \hline
 $G_{\textcolor{black}{opt}} A$ & 0 & $k n$ & \cmark & $1 - \lambda$ & 0 \\
 \hline
 $H_{\textcolor{black}{opt}} G_{\textcolor{black}{opt}} A$ & 0 & $k n$ & \cmark & $(1 - \lambda)\lambda$ & 0  \\
 \hline
 $(I + H_{\textcolor{black}{opt}}) G_{\textcolor{black}{opt}} A$ & 0 & $k n$ & \cmark & $1 - \lambda^2$ & 0  \\
 \hline
 $_{\textcolor{black}{opt}}M_{m}\iv A$ & 0 & $k n$ & \cmark & $1 - \lambda^m$ & 0  \\
 \hline
\end{tabular}\label{tab:N-even-GH-a=1-b=-1}
}
\end{subtable}
\begin{subtable}{\linewidth}
\vspace{0.5em}
\caption{$N = 2k-1$ \edit{(odd)} }
\vspace{-0.5em}
\scalebox{0.92}{
\begin{tabular}{ |c|c|c|c|c|c| } 
 \hline
 \multirow{2}{*}{Matrix} & \multirow{2}{*}{\makecell{\# Exact \\ Zeros}} & \multicolumn{3}{c|}{Eigenvalues in $(0, 1)$} & \multirow{2}{*}{\makecell{\# Exact \\ Ones}} \\ 
 \cline{3-5}
 & & \# Distinct & Pair? & Example & \\
 \hline
 $B_l\iv C_l$ & $k n$ & $(k-1) n$ & \xmark & $\lambda$ & 0 \\ 
 \hline
 $H_{\textcolor{black}{opt}}$ & $n$ & $(k-1) n$ & \cmark & $\lambda$ & 0 \\ 
 \hline
 $G_{\textcolor{black}{opt}} A$ & 0 & $(k-1) n$ & \cmark & $1 - \lambda$ & $n$  \\
 \hline
 $H_{\textcolor{black}{opt}} G_{\textcolor{black}{opt}} A$ & $n$ & $(k-1) n$ & \cmark & $(1 - \lambda)\lambda$ & 0  \\
 \hline
 $(I + H_{\textcolor{black}{opt}}) G_{\textcolor{black}{opt}} A$ & 0 & $(k-1) n$ & \cmark & $1 - \lambda^2$ & $n$  \\
 \hline
 $_{\textcolor{black}{opt}}M_{m}\iv A$ & 0 & $(k-1) n$ & \cmark & $1 - \lambda^m$ & $n$  \\
 \hline
\end{tabular}\label{tab:N-odd-GH-a=1-b=-1}
}
\end{subtable}
\vspace{-1.5em}
\end{table}

\subsection{Parametric polynomial preconditioner}
In this section we build on our rigorous analysis of the spectrum of $_{ab}M_m\iv A$ and positive definiteness of $_{ab}M_m\iv$, and introduce a common practice to further accelerate convergence of PCG. We do so through the polynomial parametrization with $\{\alpha_i\}$ where each $\alpha_i \in \mathbb{R}$ can be selected to minimize the condition number via Chebyshev iteration~\cite{johnson1983polynomial}: 
\begin{equation}\label{eq:Mminv-polynomial-coefficients}
    {}^{\alpha}_{ab}M_m\iv = (I + \alpha_1 H_{ab} + \dots + \alpha_{m-1}H_{ab}^{m-1}) G_{ab}.
\end{equation} 
Unfortunately, while such an approach is known to improve overall PCG performance, the best choice of $\{\alpha_i\}$, i.e. the one which minimizes the total number of PCG iterations, \textit{is often NOT the one that minimizes the condition number and is impossible to compute}~\cite{saad1985practical}. 
Fortunately, our previous analysis still guarantees the correctness of the following theorem, regardless of the choice of $\{\alpha_i\}$, enabling its use in practice. 
\begin{theorem}\label{thm:MoptinvA-distinct-eigenvalues}
    If $a = 1, b = -1$, then ${}^{\alpha}_{ab}M_m\iv A$ has $\frac{N}{2}n$ or $\floor*{\frac{N}{2}}n + 1$ distinct eigenvalues, depending on the parity of $N$. 
\end{theorem}
\begin{proof}
    Follows from~\Cref{thm:optimal-spectrum,lem:eigenvalues-Mabinv*A}. 
\end{proof}

\section{Time complexity analysis}
In this section, we analyze the theoretical time usage of solving $Ax = b$ directly or iteratively, where $A$ is symmetric positive definite block tridiagonal as defined in~\eqref{eq:def-spd-block-tri}. 

\subsection{Complexity of preconditioner computation}
We present the closed-form formula for $G_{opt}$ when $N=3$ and $H_{opt}$ when $N=5$. 
The cases for $a\in [0,1)$ are omitted due to space limit but can be deducted similarly.  
\begin{align}\label{eq:GH-opt-closed-form}
    & \begin{aligned}
        E_i &\coloneq D_i\iv O_i D_{i+1}\iv, \\
        i &= 1, 2, \dots, N-1
    \end{aligned} \quad
    G_{opt} = \begin{bsmallmatrix}
    D_1\iv & -E_1 & \\
    -E_1\tp & D_2\iv & -E_2 \\
    & -E_2\tp & D_3\iv
    \end{bsmallmatrix},
    \\
    H_{opt} &= \begin{bsmallmatrix}
    E_1 O_1\tp & 0 & E_1 O_2 & 0 & 0 \\
    0 & E_1\tp O_1 + E_2 O_2\tp & 0 & E_2 O_3 & 0\\
    E_2\tp O_1\tp & 0 & E_2\tp O_2 + E_3 O_3\tp & 0 & E_3 O_4 \\
    0 & E_3\tp O_2\tp & 0 & E_3\tp O_3 + E_4 O_4\tp & 0 \\
    0 & 0 & E_4\tp O_3\tp & 0 & E_4\tp O_4
    \end{bsmallmatrix}. \notag
\end{align}
\eqref{eq:GH-opt-closed-form} requires $\mathcal{O}(Nn^3)$ floating point operations due to the inevitable of dense matrix inversion $D_i\iv$. 

\begin{figure}[!t]
    \centering
    \vspace{-5pt}
    \includegraphics[width=0.8\linewidth]{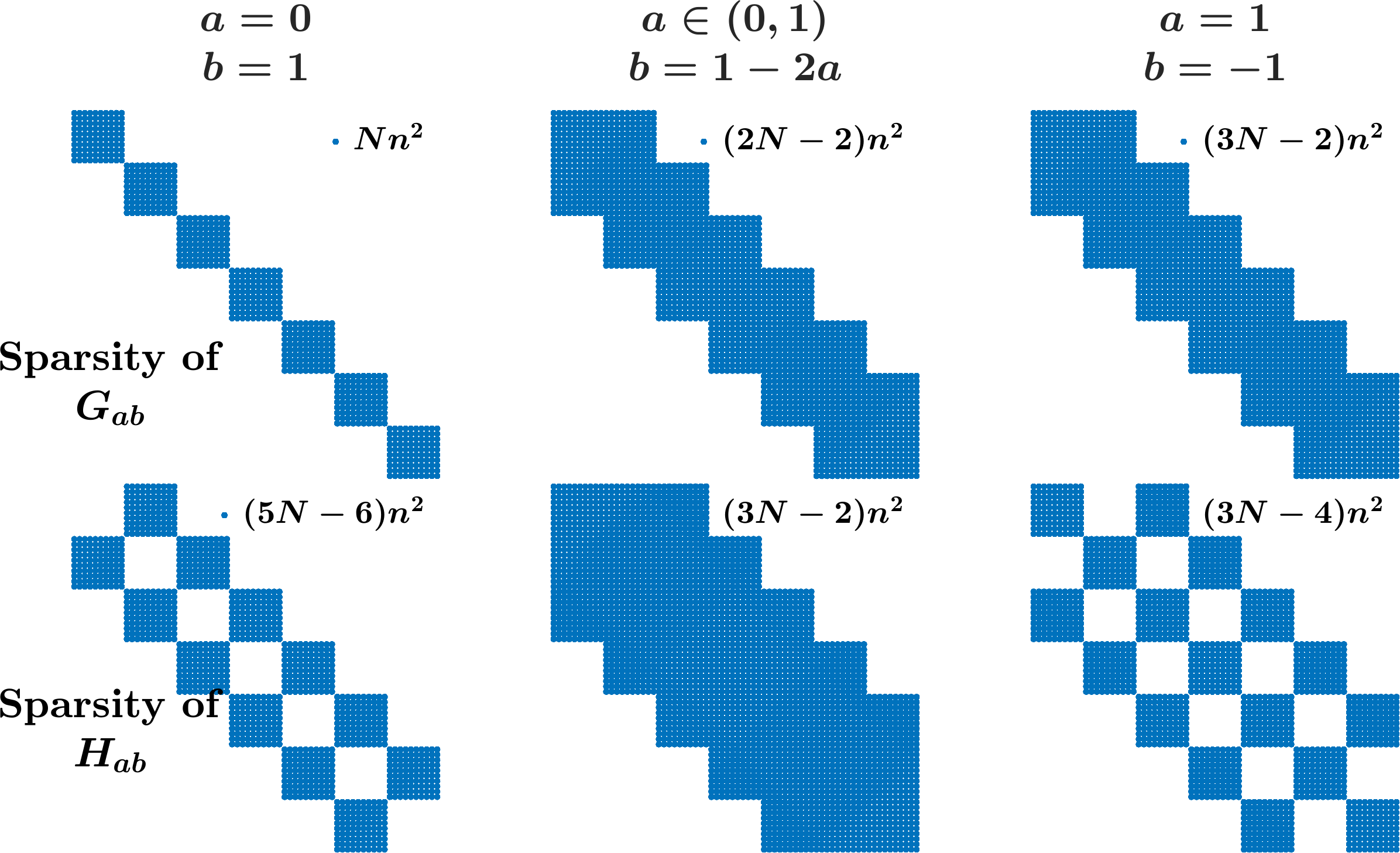}
    \caption{
    \edit{Sparsity pattern of $G_{ab}$ and $H_{ab}$ for different $(a,b)$ with $N = 7, n = 10$. 
    The counts of nonzero entries of matrices are labeled on the top-right. 
    $G_{ab}$ is block tridiagonal with the exception at $a=0$ (block Jacobi). 
    $H_{ab}$ is block pentadiagonal with the exceptions at $a=0$ and $a=1$ (optimal). }
    }
    \label{fig:sparsity}
\vspace{-1.2em}
\end{figure}

\subsection{Complexity of single iteration of PCG}
The computation of $_{ab}M_m\iv r$ and $Ap$ dominate the complexity of each iteration of PCG, where $r$ denotes the residual vector and $p$ denotes the directional vector. 
However, explicit computation of $_{ab}M_m\iv$ destroys sparsity, so it is preferable to store $(G_{ab}, H_{ab})$ instead. 
$_{ab}M_m\iv r$ can then be computed via $y_0 = G_{ab} r$, $y_i = H_{ab} y_{i-1}, i=1, \dots, m-1$. 
Depending on the choice of $a$ and polynomial order $m$, the total floating point operations to perform block-wise (band) matrix-vector multiplications ($_{ab}M_m\iv r, Ap$) is $\mathcal{O}(N n^2 g_a(m))$ where $g_a(m)$ is determined by the sparsity patterns in~\Cref{fig:sparsity}.
\begin{equation}
    g_a(m) \coloneq \begin{cases}
        2m + 2, & \text{ if $a=0$.} \\
        5m + 1, & \text{ if $a \in (0,1)$.} \\
        3m + 3, & \text{ if $a=1$.}
    \end{cases}
\end{equation}

\subsection{Complexity comparison: direct v.s. iterative}

The time complexity of directly solving~\eqref{eq:def-spd-block-tri} via Cholesky decomposition is $\mathcal{O}(N n^3)$ due to the sequential operations w.r.t. $N$, regardless of hardware usage. Specialized linear algebra kernels can reduce the $n^3$ constant but the cubicity persists. The forward/backward substitutions are neglected because they contribute $n^2$ terms. 

In contrast, since block-wise matrix-vector multiplication and preconditioner computation (e.g.~\eqref{eq:GH-opt-closed-form}) are both parallelizable, the time complexity of PCG with $_{ab}M_m\iv$ on a GPU is 
$\mathcal{O}\Big( n_{itr} \frac{N}{n_{blk}} \frac{n^2}{n_{thr}} g_a(m) + \frac{N}{n_{blk}} n^3 \Big)$, 
where $n_{itr}$ denotes the number of PCG iterations, which is bounded by $d_{Nna}$ in~\eqref{eq:distinct-eigenvalues-count} under exact arithmetic~\cite{nocedal1999numerical}. 

$n_{blk}$ and $n_{thr}$ represent the two basic parallelism hierarchies on a GPU: the number of blocks and the number of threads per block. 
Each block of threads can efficiently, in parallel, compute the product of one block row of banded matrix-vector multiplication and each $D_i\iv, E_i, E_i O_{i+1}$ in~\eqref{eq:GH-opt-closed-form}. 
Data dependencies in matrix inversion limit parallelism in those primitive operations. 
It is thus theoretically motivated and practically feasible\footnote{\edit{The latest mobile NVIDIA Jetson AGX Orin has $16$ Streaming Multiprocessor which each supports several tens of blocks of hundreds of threads. This provides more than sufficient parallelism for OCP on robotic applications.}}
to conclude: 
\begin{align}\label{eq:complexity-comparison}
    \text{If } n_{thr} = \mathcal{O}(n) \text{ and } n_{blk} &= \mathcal{O}(N) \text{, then } \\
    \mathcal{O}\Big( n_{itr} \frac{N}{n_{blk}} \frac{n^2}{n_{thr}} g_a(m) + \frac{N}{n_{blk}} n^3 \Big) &< \mathcal{O}(N n^3) \text{ for } N \gg 1. \notag
\end{align}


\color{black}
\section{Numerical results}
In this section, we present numerical results on the $m$-step polynomial preconditioner based on the proposed family of multi-splittings.
We construct random s.p.d. block tridiagonal matrices by viewing the classical LQR problem as a Quadratic Program (QP), formulating \edit{its} KKT system, and computing the Schur complement w.r.t. the Hessian as done in current \edit{GPU-accelerated} parallel solvers~\cite{adabag2024mpcgpu}. 
We \edit{then} evaluate the \edit{minimum, maximum, and number of distinct eigenvalues, and the relative condition number of the preconditioned system. 
We also collect counts of PCG iteration $n_{itr}$ and block-wise matrix-vector multiplication (\texttt{gemv}) over all $n_{itr} \cdot g_a(m)$ in MATLAB (exit condition: $\norm{Ax-b}_2 < 1\mathrm{e}{-6}$). The \texttt{gemv} count is proportional to wall clock time.}

\Cref{fig:x=polyOrderm} shows \edit{all results} for varying polynomial orders \edit{$m \in \{1,2,3,4\}$} and parameters $a, b$. 
\edit{``Block Jacobi'' points to~\Cref{subsubsec:block-jacobi}, ``Equal'' points to equal weights of diagonal and stair splittings, ``Stairs only'' points to~\Cref{subsubsec:symmetric-stair}, ``Optimal'' points to~\Cref{subsubsec:optimal}, and ``PolyCoeff'' uses $(G_{opt}, H_{opt})$ with polynomial coefficients\footnote{\edit{We set $\alpha_{m-1}=7, \alpha_i = 1, \forall i\neq m-1$ for empirical performance using a grid search. No $\alpha$ is involved if $m=1$, so it is the same as ``Optimal''. }} as in~\eqref{eq:Mminv-polynomial-coefficients}. 
For each $(a,b,m)$ triple, $100$ matrices $A$ are randomly generated. For each $A$, $Ax = b$ is solved by PCG with $100$ randomly generated vectors $b$.}
The condition numbers are normalized to the ``\edit{Block Jacobi}'' + ``$m=1$'' to enable \edit{comparisons across settings which substantially differ in absolute value.} 

\edit{Ignoring ``PolyCoeff'' temporarily, the first three subplots of~\Cref{fig:x=polyOrderm} validate the claims of~\Cref{thm:optimal-spectrum}: $_{opt}M_m\iv A$ has the most clustered spectrum: least distinct eigenvalues and narrowest spectrum range. 
The \nth{4} and \nth{5} subplots prove that}
for all pairs of $a, b$, larger $m$ leads to smaller PCG iteration counts and condition number, as predicted in~\Cref{rmk:big-m-good-preconditioner}. 
\edit{In all subplots, ``Optimal''+``$m=1,2$'' are the same as ``Block Jacobi''+``$m=2,4$'', as predicted in~\cref{ft:inspection}.}

\edit{The \nth{6} subplot delivers the key message: ``Optimal'' is faster than ``Block Jacobi'' when $m$ is odd but slower when $m$ is even. With the help of~\eqref{eq:Mminv-polynomial-coefficients}, ``PolyCoeff'' then requires the least amount of block \texttt{gemv} operations for all $m$ and achieves minimum at $m=2$. 
Its success originates from: 
1) The eigenvalues of ${}^{\alpha}_{opt}M_m\iv A$ always come in pairs, as proven by~\Cref{thm:MoptinvA-distinct-eigenvalues}, a fact that does not generalize to other $(a, b) \in \mathcal{C}_g$, especially the competing case $a=0$; and
2) the polynomial parametrization reduces the condition number (\nth{4} subplot) at the cost of slightly wider spectrum range (\nth{1} and \nth{2} subplots). Both points reduce the upper bound of the number of PCG iterations. 
At $m=2$, the only difference between ``Block Jacobi'' and ``Optimal/PolyCoeff'' is the \nth{2} line of~\eqref{eq:GH-opt-closed-form}. The computation time of $H_{opt}$ is $\mathcal{O}(\frac{N}{n_{blk}} \frac{n^3}{n_{thr}} )$. 
For large $N$ and $n$, such time can be compensated by the difference between the blue and green bars in \nth{6} subplot. 
}
\begin{figure}[!t]
    \centering
    \includegraphics[width=0.99\linewidth]{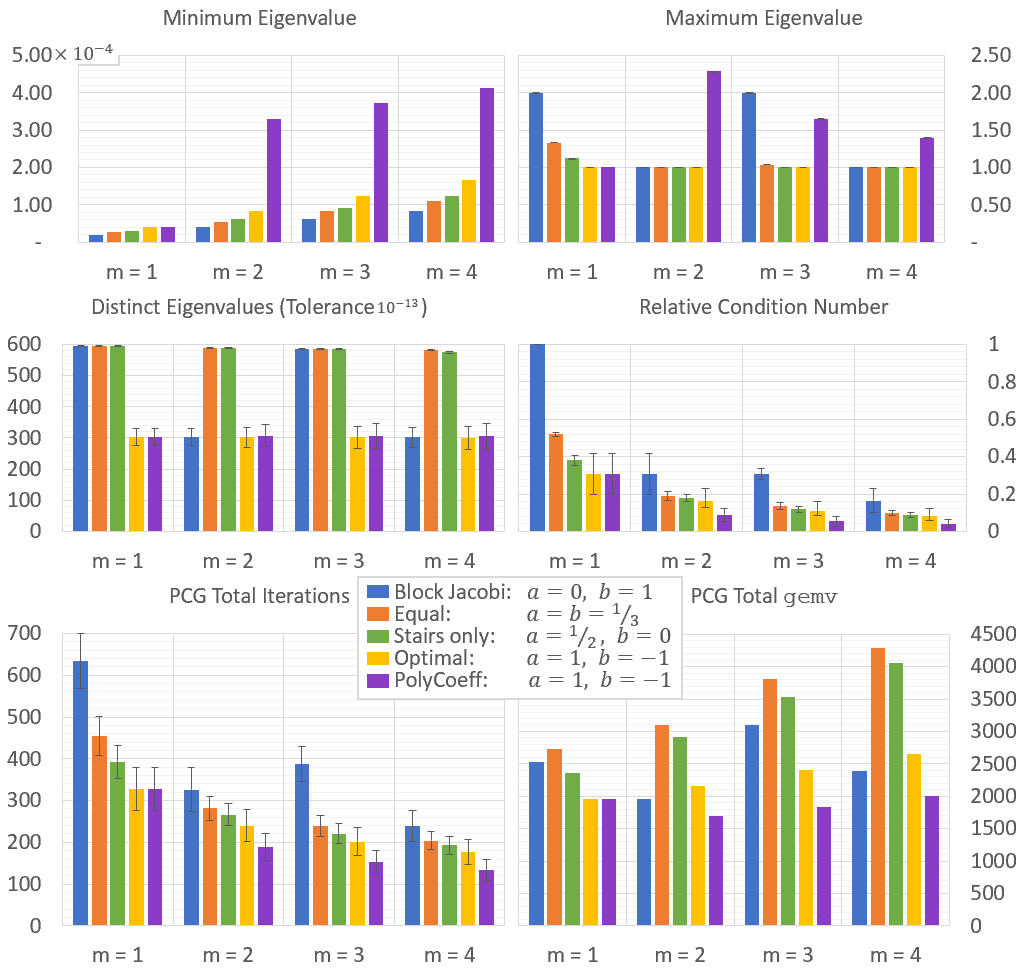}
    \vspace{-15pt}
    \caption{\edit{Statistics for different $(a,b)$ with randomly generated s.p.d. block tridiagonal matrices $N = 30, n = 20$.}}
    \label{fig:x=polyOrderm}
\vspace{-15pt}
\end{figure}

\section{Conclusion}
\edit{We develop a parametric} family of $m$-step polynomial preconditioners tailored for symmetric positive definite block tridiagonal matrices. Building on the findings of~\cite{bu2024symmetric}, we extend the base case ($m=1$) to a general case of $m \in \mathbb{N}$, while preserving parallel efficiency. 
We \edit{provide} necessary and sufficient conditions \edit{for positive definiteness of the preconditioner, qualifying its usage for PCG, and}
demonstrate a unique set of optimal parameters \edit{and polynomial coefficients that achieve} the most clustered spectrum, \edit{the fewest distinct eigenvalues, and the best resulting PCG performance.
}

In future work, we aim to further reduce PCG iterations by parametrizing\edit{~\eqref{eq:Mminv-polynomial-coefficients} properly rather than grid search. 
We will conduct numerical comparison of PCG against direct factorization to validate the theoretical claim~\eqref{eq:complexity-comparison}. 
We also aim to leverage the proposed preconditioner to both improve upon existing parallel (S)QP solvers for OCPs~\cite{adabag2024mpcgpu} 
and for applications in other scientific domains (e.g. solving PDE).} 

\bibliographystyle{IEEEtran}
\bibliography{IEEEabrv, main.bib}

\appendix
\subsection{Proof of \texorpdfstring{\Cref{lem:a-b-negative-rhoH-spectrum-of-GabA}}{Theorem}}\label{appendix:A-proof}

\begin{table}[!t]
    \centering
    \caption{
    \edit{Summary of the proof of~\Cref{lem:a-b-negative-rhoH-spectrum-of-GabA} if $a \in [0,1]$}
    }
    \vspace{-0.5em}
\begin{subtable}{\linewidth}
    \caption{Extreme points of $f_{a-}(\lambda)$ for $\lambda \in (0,1)$}
    \vspace{-0.5em}
    \scalebox{0.78}{
    \begin{tabular}{|c|c|c|c|c|c|c|}
        \hline
        \multirow{2}{*}{$a$}& \multicolumn{3}{c|}{MAX} & \multicolumn{3}{c|}{MIN} \\
        \cline{2-7} 
        & arg & reachable? & value & arg & reachable? & value \\
        \hline 
        $[0,\frac{1}{3}]$ & $0$ & $\xmark$ & $0$ & $1$ & $\xmark$ & $2a-1 \in [-1, -\frac{1}{3})$  \\
        \hline
        $(\frac{1}{3}, \frac{1}{2}]$ & $0$ & $\xmark$ & $0$ & $\lambda^*$ & $\cmark$ & $\frac{(1-a)^2}{-4a} \in (-\frac{1}{3}, -\frac{1}{8}]$ \\
        \hline
        $(\frac{1}{2}, 1]$ & $1$ & $\xmark$ & $2a-1 \in (0,1]$ & $\lambda^*$ & $\cmark$ & $\frac{(1-a)^2}{-4a} \in (-\frac{1}{8}, 0]$ \\
        \hline
    \end{tabular}
    }
    \label{tab:extreme-points}
\end{subtable}
\begin{subtable}{\linewidth}
    \vspace{0.5em}
    \caption{Range of $f_{a-}(\lambda)$, $1-f_{a-}(\lambda)$, and spectrum of $H_{ab}, G_{ab}A$}
    \vspace{-0.5em}
    \scalebox{0.81}{
    \begin{tabular}{|c|c|c|c|c|}
        \hline
        $a$ & range of $f_{a-}$ & $\mathcal{I}_{H_{ab}}$ & range of $1-f_{a-}$ & $\mathcal{I}_{G_{ab}A}$ \\
        \hline 
        $[0,\frac{1}{3}]$ & $(2a-1, 0)$ & $(2a-1, 1)$ & $(1, 2-2a)$ & $(0, 2-2a)$ \\
        \hline 
        $(\frac{1}{3}, \frac{1}{2}]$ & $[\frac{(1-a)^2}{-4a}, 0)$ & $[\frac{(1-a)^2}{-4a}, 1)$ & $(1, 1+\frac{(1-a)^2}{4a}]$ & $(0, 1+\frac{(1-a)^2}{4a}]$ \\
        \hline 
        $(\frac{1}{2}, 1]$ & $[\frac{(1-a)^2}{-4a}, 2a-1)$ & $[\frac{(1-a)^2}{-4a}, 1)$ & $(2-2a, 1+\frac{(1-a)^2}{4a}]$ & $(0, 1+\frac{(1-a)^2}{4a}]$ \\
        \hline
    \end{tabular}
    }
    \label{tab:range-of-spectrum}
\end{subtable}
\vspace{-2em}
\end{table}

\begin{proof}
\edit{Recall $f_{a+}(\lambda) = a \lambda + (1-a)\sqrt{\lambda}$ and $f_{a-}(\lambda) = a \lambda + (a-1)\sqrt{\lambda}$}. 
\begin{enumerate}[leftmargin=*]
\item If $a \in [0, 1]$, then $f_{a+}$ is monotonically increasing. 
$f_{a+}(0)=0, f_{a+}(1)=1 \Rightarrow f_{a+} \in (0, 1)$. 
\edit{So is $1-f_{a+}$. }
\begin{itemize}
    \item The extreme points of $f_{a-}$ are summarized in~\Cref{tab:extreme-points}. 
    \begin{itemize}
        \item If $a = 0$, \edit{then} $f_{a-} = -\sqrt{\lambda} \in (-1, 0)$ so it is monotonically decreasing. 
        \item  If $a \in (0, 1]$, \edit{then $f_{a-}' = 0$ at $\lambda^* = \frac{(1-a)^2}{4a^2} \geq 0$.} 
        \edit{Its extreme point has three candidates}: $f_{a-}(0) = 0$, $f_{a-}(1) = 2a-1$, and $f_{a-}(\lambda^*) = -\frac{(1-a)^2}{4a}$.
        \item If $\lambda^* \in (0,1)$, \edit{then} $ \frac{(1-a)^2}{4a^2} < 1 \Rightarrow a > \frac{1}{3}$.
        $f_{a-}$ decreases on $(0, \lambda^*)$ and increases on $(\lambda^*, 1)$. So $f_{a-}$ achieves its maximum at $\lambda = 0$ or $\lambda = 1$ (depends on whether $a \geq \frac{1}{2}$) and minimum at $\lambda^*$. 
        \item $a \leq \frac{1}{3} \Rightarrow \lambda^* \geq 1$, $f_{a-}$ decreases on $(0, 1)$, so the maximum is at $\lambda = 0$ and minimum at $\lambda = 1$.
    \end{itemize}
    \item $\rho(H_{ab}) < 1 \Leftrightarrow \mathcal{I}_{H_{ab}} \subset (-1, 1)$ is concluded from the third column of~\Cref{tab:range-of-spectrum}. $\mathcal{I}_{G_{ab}A} \subset \mathbb{R}_{>0}$ is concluded from the fifth column. 
\end{itemize}
\item If $a \in (1, +\infty)$, \edit{then} $f_{a-}$ is monotonically increasing because \edit{$a, a-1 > 0$}. 
$\frac{1}{a} \in (0,1)$ holds. $f_{a-}(\frac{1}{a}) = 1 + (a-1)\sqrt{\frac{1}{a}} > 1$, so $\forall \lambda \in [\frac{1}{a}, 1)$ leads to $\rho(H_{ab}) > 1$ and negative eigenvalues of $G_{ab}A$. 
\item If $a \in (-\infty, 0)$, \edit{then} $f_{a-} < 0$ is monotonically decreasing because \edit{$a, a-1 < 0$}. $\frac{1}{(1-a)^2} \in (0,1)$ holds. $f_{a-}(\frac{1}{(1-a)^2}) = \frac{a}{(1-a)^2} -1 < -1$, so $\forall \lambda \in [\frac{1}{(1-a)^2}, 1) $ leads to $\rho(H_{ab}) > 1$. \\
$f_{a-}(\lambda) < 0 \Rightarrow 1- f_{a-}(\lambda) > 0$, so the signs of eigenvalues of $G_{ab}A$ depend on $f_{a+}$. 
\edit{$f_{a+}' = 0$ at $\lambda^* = \frac{(1-a)^2}{4a^2} > 0$. }
\begin{itemize}
    \item If $\lambda^* \in (0,1)$, then $\frac{(1-a)^2}{4a^2} < 1 \Rightarrow a < -1$, 
    $f_{a+}$ increases on $(0, \lambda^*)$, decreases on $(\lambda^*, 1)$, and achieves its maximum at $f_{a+}(\lambda^*) = -\frac{(1-a)^2}{4a} > 1 \Rightarrow$ $G_{ab}A$ has negative eigenvalue. 
    \item If $a \geq -1$, \edit{then} $\lambda^* \geq 1$, $f_{a+}$ is monotonically increasing on $(0,1)$. So $f_{a+}(\lambda)$ and $1-f_{a+}(\lambda) \in (0, 1)$. \edit{Meanwhile}, $f_{a-} \in (2a-1, 0)$ and $1-f_{a-} \in (1, 2-2a)$ coincide with the first rows of~\Cref{tab:extreme-points,tab:range-of-spectrum}. 
\end{itemize}
\end{enumerate}
All points are proven. 
\end{proof}

\end{document}